\documentclass[12pt,reqno,a4paper]{amsart}
\usepackage{
    amsmath,  amsfonts, amssymb,  amsthm,   amscd,
    gensymb,  graphicx, comment,  etoolbox, url,
    booktabs, stackrel, mathtools,enumitem, mathdots,  microtype, lmodern,    mathrsfs, graphicx, tikz,  longtable,tabularx, float, tikz, pst-node, tikz-cd, multirow, tabularx, amscd,  bm, array, makecell, diagbox, booktabs,ragged2e, caption, subcaption }
\usepackage{makecell,slashbox}
\usepackage{algorithm}
\usepackage{algpseudocode}
\usepackage{esvect}

\MakeRobust{\vv}
\usepackage{xcolor}
\usepackage[utf8]{inputenc}
\usepackage{fullpage, wrapfig,textcomp, csquotes,fbb}
\usepackage[pagebackref=true, colorlinks=true, linkcolor=blue, citecolor=blue, urlcolor=blue, breaklinks=true]{hyperref}
\usepackage[capitalise]{cleveref}
\usepackage{todonotes}
\usetikzlibrary{positioning}
\usetikzlibrary{shapes,arrows.meta,calc}
 
\usetikzlibrary{arrows}
\definecolor{omegaOne}{RGB}{180, 215, 255}    
\definecolor{omegaTwo}{RGB}{255, 215, 180}    
\definecolor{omegaThree}{RGB}{180, 235, 180}

\newtheorem{theorem}{Theorem}[section]

\newtheorem{corollary}[theorem] {Corollary}
\newtheorem{definition}[theorem]{Definition}
\newtheorem{example}[theorem]{Example}
\newtheorem{lemma}[theorem]{Lemma}

\newtheorem{proposition}[theorem]{Proposition}
\newtheorem{remark}[theorem]{Remark}

\crefalias{corollary}{corollary}
\crefalias{theorem}{theorem}
\crefalias{lemma}{lemma}
\crefalias{proposition}{proposition}
\crefalias{definition}{definition}
\newcommand{\scat}{\mathrm{scat}}

\newcommand{\dcscott}{\mathrm{DC}^{Sc}}
\newcommand{\dcMoriloWu}{\mathrm{DC}^{MW}}
\newcommand{\dcscottspace}{\mathrm{DC}}

\newcommand{\TC}{\mathrm{TC}}
\newcommand{\Tcscott}{\mathrm{TC^{Sc}}}

\newcommand{\ssec}{\mathrm{s}\text{-}\mathrm{secat}}

\newcolumntype{x}[1]{>{\centering\arraybackslash}p{#1}}

\begin{document}
\title{Discrete version of topological complexity of maps}

\author[S. Datta ]{Sutirtha Datta}
\address{Department of Mathematics, Indian Institute of Science Education and Research Pune, India}
\email{sutirtha2702@gmail.com}
\author[N. Daundkar ]{Navnath Daundkar}
\address{Department of Mathematics, Indian Institute of Technology Madras, India}
\email{navnath@iitm.ac.in}
\author[A. Sarkar ]{Abhishek Sarkar}
\address{School of Computing, MIT Art, Design and Technology University, Loni Kalbhor, Pune, India}
\email{abhishek.sarkar@mituniversity.edu.in}
\author[S. Sarkar]{Soumen Sarkar}
\address{Department of Mathematics, Indian Institute of Technology Madras, Chennai-600036, India}
\email{soumen@iitm.ac.in}
 
\begin{abstract}
We introduce and study discrete analogs of Scott's and Murillo-Wu's topological complexity of  maps. We prove that these discrete analogs are  contiguity invariants and are, in fact, equivalent. Furthermore, we establish the fundamental theoretical properties and computational aspects of the discrete topological complexity of simplicial maps.

\end{abstract}
\maketitle
\section{Introduction}
Topological complexity, $\mathrm{TC}(X)$, was introduced by Farber \cite{FarberTC} to quantify the navigational complexity of a robot's configuration space $X$. Defined as the sectional category (or Schwarz genus) of the path space fibration $\pi_X : PX \to X \times X$, it serves as a homotopy invariant classically estimated via Lusternik-Schnirelmann (LS) category and cohomological cup-length (see \cite{Farberbook}). To facilitate combinatorial analysis and algorithmic implementations, Fernández-Ternero et al. \cite{fernandez2018discrete} established a discrete analog of topological complexity in the setting of abstract simplicial complexes. There is another notion of simplicial topological complexity using barycentric subdivisions due to Jes\'us Gonz\'alez, see \cite{gonzalez2017simplicial}. However, we will restrict ourselves to the discrete topological complexity setting. For higher analogs of discrete topological complexity, see \cite{alabay2024higher}.

For broader motion planning scenarios where tasks are constrained by a target workspace, several  topological complexity variants of map-level have been proposed, most notably by Pave\v{s}i\'c \cite{pavevsic2018topological}, Murillo--Wu \cite{murillo2021topological}, and Scott \cite{scott2022topological}. For a work map $f : X \to Y$ mapping configuration states to workspace goals, $\mathrm{TC}(f)$ measures the joint complexity of motion planning relative to the task. While Pave\v{s}i\'c's formulation is not homotopy invariant in general, Scott and Murillo--Wu introduced homotopy-invariant definitions that were shown to be equivalent in the continuous setting.

In this paper, we extend the Scott and Murillo--Wu frameworks to the combinatorial setting by introducing the discrete topological complexity of simplicial maps. We adapt Scott's model to the setting of abstract simplicial complexes by introducing $f$-Farber subcomplexes, defining the discrete invariant $\mathrm{DC}^{\mathrm{Sc}}(f)$ and \Cref{prop: contiguity invariance of dcscottmap} proves that it is invariant under contiguous simplicial maps. Finally, in \Cref{prop: Sc=MW} we establish the equivalence between our Scott-type discrete topological complexity and the discrete analog of the Murillo--Wu formulation. This work complements recent results of {\.I}s and Karaca \cite{is2024discrete}, who studied a discrete analog of Pave\v{s}i\'c's fibration-based complexity.

The layout of the paper is as follows: In \Cref{section 2 : Background}, notions of simplicial LS category, simplicial sectional category and discrete topological complexity and the topological complexity of maps are recalled. \Cref{section 3: scott} is devoted to defining the discrete version of Scott's topological complexity and establishing its contiguity invariance property. In \Cref{section 4 M-W} we define the discrete version of Murillo-Wu's topological complexity of maps and prove the equivalence of the two notions in \Cref{prop: Sc=MW}.

\section{Background: contiguity invariants and topological complexity of maps}\label{section 2 : Background}
In this section, we recall several key notions, including simplicial LS category of abstract simplicial complexes and simplicial maps, simplicial sectional category, and discrete topological complexity. All simplicial complexes are assumed to be finite and (edge-path) connected, and all maps between simplicial
complexes are assumed to be simplicial maps. The simplicial map  $K \to L$ that sends all of $K$ to some vertex in $L$ will be denoted by $*.$

We first revisit the notion of contiguity, considered as the discrete version of homotopy in the case of simplicial complexes and simplicial maps. Given two simplicial complexes $K$ and $L$, a pair of simplicial maps $\phi,\psi: K \to L$ are called \textit{contiguous} if, for any simplex $\sigma \in K$, the set $\phi(\sigma) \cup \psi(\sigma)$ is a simplex of $L$ and will be denoted by $\phi \sim_c \psi.$ A pair of simplicial maps $\phi, \psi: K \to L$ are said to be in the \textit{same contiguity class} if there is a finite sequence $$\phi:=\phi_0 \sim_c \phi_1\sim_c \dots \sim_c \phi_n := \psi$$ of simplicial maps $\phi_i : K \to L$ for $0 \leq i \leq n$.  
\subsection{Simplicial LS category} 
The notion of strong homotopy type, strong collapse, and expansion was introduced and studied by Barmak and Minian in \cite{barmak2009algebraic, barmak2012strong}, which was later used by Fernández-Ternero et al. in \cite{fernandez2015lusternik} to define the corresponding notion of LS-category for simplicial complexes. By \cite[Corollary 2.12]{barmak2012strong}, two simplicial complexes $K$ and $L$ have the same strong homotopy
type if and only if there are simplicial maps $\phi: K \to L$ and $\psi: L \to K$ such that
$\psi \circ \phi \sim id_K$ and $\phi \circ \psi \sim id_L.$
\begin{definition}
    Let $K$ be an abstract simplicial complex. The (normalized) simplicial LS-category of $K$, denoted by $\operatorname{scat}(K)$, is the least integer $m \geq 0$ such that there exist subcomplexes 
$$L_0, L_1, \dots, L_m \subset K \quad\text{such that} \quad K = \bigcup_{i=0}^{m} L_i,$$
with each $L_i$ being categorical. That is, for every $i$, the inclusion map
\[
\iota_{L_i} : L_i \hookrightarrow K
\]
lies in the contiguity class of a constant map $L_i \to K$, i.e.,  $\iota_{L_i} \sim *$.
\end{definition}

The notion of simplicial LS-category of simplicial maps was introduced by Scoville and Swei in \cite[Definition 7]{S-S}.
\begin{definition}
    Let $f:K \rightarrow L$ be a simplicial map between two abstract simplicial complexes $K$ and $L$. A subcomplex $J \subset K$ is said to be $f$-categorical if the map $f|_J:J \rightarrow L$ belongs to the contiguity class of some constant map $J \rightarrow L$, i.e., $f|_J \sim *$. The minimum number $m \geq 0$ such that there are $f$-categorical subcomplexes $J_0, \dots, J_m$ which cover $K$ is called the simplicial \emph{LS}-category of the map $f$, denoted by $\scat(f)$.
    \end{definition}
    \begin{remark}
     For a simplicial complex $K$, $\scat(Id_K)=\scat(K)$. 
    \end{remark} 
\subsection{Simplicial sectional category}
In the classical setting the LS-category can be viewed as the sectional category of a path fibration. In the simplicial setting the notion of sectional category was introduced by Fernández-Ternero et al. in \cite[Definition 18]{fernandez2021simplicial}.
\begin{definition}
    Let \( f : K \to L \) be a simplicial map. The simplicial sectional category (or \v{S}varc genus) of \( f \), denoted by $\ssec(f)$ is the smallest integer \( n \geq 0 \) such that \( L \) can be expressed as a union
\[
L = L_0 \cup \cdots \cup L_n
\]
of subcomplexes, and for each \( j \) there exists a section \( \sigma_j \) of \( f \), that is, a simplicial map
$$\sigma_j : L_j \to K~~\text{satisfying}~~f \circ \sigma_j = \iota_j : L_j \hookrightarrow L.$$
If no such $n \geq 0$ exists, we define $\ssec(f)= \infty$.
\end{definition}
 
\subsection{Discrete topological complexity}\label{discrete TC} First, we recall the definition of  Farber subcomplexes and discrete topological complexity defined in \cite{fernandez2018discrete}.

\begin{definition}\label[definition]{defn: farber subcomplex}
Let $K$ be a simplicial complex and let $K^2 = K \Pi K$ denote its categorical product. A simplicial subcomplex $\Omega \subseteq K^2$ is called a \emph{Farber subcomplex} if there exists a simplicial map
\[
\sigma : \Omega \to K ~~\text{such that}~~ \Delta_K \circ \sigma \sim \iota_{\Omega},
\]
where $\Delta_K : K \to K^2$ is the diagonal map given by $\Delta_K(v) = (v,v)$, and $\iota_{\Omega} : \Omega \hookrightarrow K^2$ is the inclusion. 
\end{definition}

\begin{definition}
The \emph{discrete topological complexity} of $K$, denoted by $\mathrm{DC}(K)$, is the least integer $n \ge 0$ such that there exist Farber subcomplexes $\Omega_0, \ldots, \Omega_n \subseteq K^2$ satisfying
\[
K^2 = \bigcup_{j=0}^{n} \Omega_j.
\]
If no such $n$ exists, we define $\mathrm{DC}(K)= \infty$.
\end{definition}

\begin{remark} In \cite[Theorem 3.3]{fernandez2018discrete} it was shown that the discrete topological complexity is an invariant of the strong homotopy type.
The invariant $\mathrm{DC}(K)$ is also referred to as the \emph{simplicial complexity} of $K$ by the authors of \cite{fernandez2018discrete}.
\end{remark}

 The reader is referred to~\cite{fernandez2018discrete} for further results, including comparisons between simplicial LS-category and discrete topological complexity (TC), as well as relationships between discrete TC and the TC of geometric realizations.

The discrete TC of wedges of circles is computed in \cite[Theorem 5.6]{fernandez2018discrete}. In particular, if $|K|=S^1$, then $\mathrm{DC}(K)=2>\TC(S^1)=1$, and if $|K|=\vee_n S^1$ for $n\geq 2$, then $\mathrm{DC}(K)=2=\TC(|K|)$.
From a computational perspective, only a limited number of examples are known beyond wedges of circles.

We establish a multiplicative product inequality for the discrete topological complexity of two simplicial complexes which gives an upper bound on the $\dcscottspace(K \Pi L)$.
\begin{proposition}\label[Proposition]{prop:multiplicative product inequality for DC}
    For simplicial complexes $K$ and $L$ we have $$\dcscottspace(K \Pi L) \leq (\dcscottspace(K)+1)(\dcscottspace(L)+1)-1.$$
\end{proposition}
\begin{proof}
    Suppose that $\dcscottspace(K)=m$ and $\dcscottspace(L)=n$, then we have the following two diagrams where $\Omega \in \{\Omega_i\}_{i=0}^{i=m} \subset K \Pi K$ and $\Omega' \in \{\Omega'_j\}_{j=0}^{j=n} \subset L \Pi L$ are Farber Subcomplexes. 
    We have the following two diagrams.
    \[\begin{tikzcd}
 & K  \arrow[d, "\Delta_K" ] \\%
 \Omega \arrow[r, hook, "i_{\Omega}"] \arrow[ur, "\sigma"] & K \Pi K,
\end{tikzcd} \quad \begin{tikzcd}
 & L  \arrow[d, "\Delta_L" ] \\%
 \Omega' \arrow[r, hook, "i_{\Omega'}"] \arrow[ur, "\sigma'"] & L \Pi L. 
\end{tikzcd}\]
    We claim that the categorical product $\Omega \Pi \Omega'$ satisfies the condition for being a Farber subcomplex of $ \bigcup\limits_{\substack{0 \leq i \leq m \\ 0 \leq j \leq n}} \Omega_i \Pi \Omega'_j=\left(K \Pi L\right)\Pi\left(K \Pi L\right) \cong K^2 \Pi L^2$, so that we have the following:
    \[\begin{tikzcd}
 & K\Pi L  \arrow[d, "\Delta_{K \Pi L}" ] \\%
 \Omega \Pi \Omega' \arrow[r, hook, "i_{\Omega \Pi \Omega'}"] \arrow[ur, "\sigma \Pi \sigma'"] & K^2 \Pi L^2. 
\end{tikzcd}\]
Indeed $i_{\Omega \Pi \Omega'} = (i_{\Omega}, i_{\Omega'})$, $\Delta_{K \Pi L} = (\Delta_K, \Delta_L)$ and $\sigma \Pi \sigma' = (\sigma, \sigma')$. Moreover, by our assumption we have $\Delta_K \circ \sigma \sim i_{\Omega}$ and $\Delta_L \circ \sigma' \sim i_{\Omega'}$. Hence $\Delta_{K \Pi L} \circ (\sigma \Pi \sigma') \sim i_{\Omega \Pi \Omega'}.$ Thus, we obtain $(m+1)(n+1)$ Farber subcomplexes formed by $\{\Omega_i \Pi \Omega'_j ~|~ 0 \leq i \leq m, 0 \leq j \leq n\}$ of $K^2 \Pi L^2$, which covers $K^2 \Pi L^2$. Thus, by definition, we obtain the required result.
\end{proof}

\begin{example}
As a consequence of Proposition \ref{prop:multiplicative product inequality for DC}, we obtain $\dcscottspace(L)\leq 8$ if the geometric realization $|L|\simeq S^1\times S^1$. Note that $\TC(|L|)=2$. Suppose that $|K|\simeq (S^1\vee S^1)\times (S^1\vee S^1)$. Recall that $\TC(|K|)=4$ and $\dcscottspace(K')=2$ if $|K'|\simeq S^1\vee S^1$. By Proposition \ref{prop:multiplicative product inequality for DC}, we have $\dcscottspace(K)\leq 8$. Therefore, using \cite[Theorem 5.2]{fernandez2018discrete} $4\leq \dcscottspace(K)\leq 8$.
\end{example}

\subsection{Topological Complexity of Maps}
The idea of topological complexity of maps in \cite{scott2022topological} was to incorporate $f$-motion planners, where $f:X \to Y$ is a continuous map. An $f$-motion planner on a subset $U \subset X \times X$ is a map $f_U: U \to PY$ such that 
$$f_U(x_0,x_1)(0)=f(x_0)~~~\text{and}~~~f_U(x_0,x_1)(1)=f(x_1).$$
\begin{definition}[{\cite[Definition 3.1]{scott2022topological}}]
     The topological complexity of $f$, denoted $\TC(f)$, is the least $k$ such that $X \times X$ can be covered by $k+1$ open subsets $U_0,\dots,U_k$ on which there are $f$-motion planners. If no such $k$ exists, we define $\TC(f)=\infty.$   
\end{definition}
There is another version of topological complexity of maps by Murillo and Wu.
\begin{definition}[{\cite[Definition 0.1]{murillo2021topological}}]
    Given a continuous map $f: X \to Y$ , the topological complexity of $f$ , $\TC(f)$, is the least integer $n$ such that $X \times X$ can be covered by $n+1$ open sets $\{U_i\}_{i=0}^n$ on each of which there is a continuous map $\sigma_i : U_i \to PX$ satisfying $$(f \times f ) \circ \pi \circ \sigma_i \simeq (f \times f)|_{U_i},$$ where $\pi: PX \to X \times X$ is the path ﬁbration, $\pi(\alpha)=(\alpha(0), \alpha(1))$.  If no such $k$ exists, we define $\TC(f)=\infty.$
\end{definition}
In both papers, several results relate the topological complexity  of maps (the notions are equivalent) to the topological complexity of the domain and codomain, composition of maps, homotopy invariance, and related properties. In subsequent sections, we introduce discrete (or combinatorial) analogs of these two notions and establish their contiguity invariance and related properties.

\section{Scott's discrete topological complexity of simplicial maps}\label{section 3: scott}
In this section, we generalize Scott's notion of topological complexity of maps \cite{scott2022topological} to the discrete setting. We establish discrete analogs of several fundamental results from the continuous setting by extending known results on the discrete topological complexity of simplicial complexes (see Section \ref{discrete TC}, and \cite{fernandez2018discrete} for more details).
 
\begin{definition}
  Let $f: K \to L$ be a simplicial map between two simplicial complexes. Then the discrete topological complexity in the sense of Scott, denoted by $\dcscott(f)$,  is the smallest integer $n$ such that there exists a cover  $\{\Omega_0,\dots, \Omega_n\}$ of $K\Pi K$ consisting of subcomplexes with simplicial maps $s_i:\Omega_i\to L$ satisfying $\Delta_L\circ s_i\sim f^2|_{\Omega_i}$ for each $0\leq i\leq n$ (see the following commutative diagram). 
\[\begin{tikzcd}
 &  & L  \arrow[d, "\Delta_L" ] \\%
 \Omega_i \arrow[r, hook, "i"] \arrow[rru, bend left, "\exists ~ s_i"] & K \Pi K \arrow[r, "f^2"] & L \Pi L.
\end{tikzcd}\]
Each such subcomplex $\Omega_i$ is called an $f$-Farber subcomplex.
\end{definition}

\begin{remark}
   Observe that when $f=\mathrm{Id}_K$, the $\Omega_i$ in the above definition are called Farber subcomplexes and moreover $\dcscott(\mathrm{Id}_K) =\dcscottspace(K)$.
\end{remark}
Next, we study the properties of $\dcscott(f)$. To begin with, we have contiguity invariance.
\begin{proposition}\label[proposition]{prop: contiguity invariance of dcscottmap}
    If $f,~g: K \rightarrow L$ are contiguous simplicial maps between simplicial complexes $K$ and $L$, then $$\dcscott(f)=\dcscott(g).$$
\end{proposition}
\begin{proof}
    Let $\dcscott(g)=n$ and let $\Omega \subset K\Pi K$ be a $g$-Farber subcomplex, in other words, we have a simplicial map $\sigma : \Omega \to L$ such that $ \Delta_L \circ \sigma \sim (g \Pi g)|_{\Omega}.$ Note that $f \sim g$ implies $f \Pi f \sim g \Pi g$. Thus, we have $\Delta_L \circ \sigma \sim (f \Pi f)|_{\Omega}$. Thus, $\dcscott(f) \leq \dcscott(g)$. The other inequality is also proved similarly. 
\end{proof}
In the following propositions, we establish analogs of several standard results for the topological complexity of maps (see \cite{scott2022topological}) in the discrete setting.
\begin{proposition}\label[proposition]{prop: dcscottmap bounded by dcspaces}
    Given a simplicial map $f: K \to L$, we have $$\dcscott(f) \leq \mathrm{min}\{\dcscottspace(K), \dcscottspace(L)\}.$$
\end{proposition}
\begin{proof}
    First, we prove that $\dcscott(f)$ is less than or equal to $\dcscottspace(K)$. Suppose that $\Omega \subset K \Pi K$ is a Farber subcomplex. Then there exists a simplicial map $\sigma : \Omega \to K$ such that $\Delta_K \circ \sigma \sim i_{\Omega}$. 
    \[\begin{tikzcd}
 & K \arrow[r, "f"]\arrow[d, "\Delta_K"] & L  \arrow[d, "\Delta_L" ] \\%
 \Omega \arrow[r, hook, "i"] \arrow[ru, bend left, " \sigma"] & K \Pi K \arrow[r, "f^2"] & L \Pi L.
\end{tikzcd}\]
Note that from the contiguity commutative square we have $\Delta_L \circ f \sim f^2 \circ \Delta_K$ and hence 
$$\Delta_L \circ f \circ \sigma \sim (f^2 \circ \Delta_K) \circ \sigma \sim f^2 \circ i_{\Omega} = (f^2)|_{\Omega}$$
Therefore, $f\circ \sigma$ acts as the required simplicial map for $\Omega$ to be an $f$-Farber subcomplex and consequently, $\dcscott(f) \leq \dcscottspace(K).$

For the other inequality, let $\Omega \subset L \Pi L$ be a Farber subcomplex. Denote $\Omega':= (f^2)^{-1}(\Omega)$. Let $\sigma:= s \circ f^2|_{\Omega'}.$ Then $$\Delta_L \circ \sigma= \Delta_L \circ s \circ f^2|_{\Omega'} \sim f^2|_{\Omega'}$$ where the last contiguity follows since $\Delta_L \circ s \sim i_{\Omega}.$ This proves that $\Omega'$ is an $f$-Farber subcomplex and hence $\dcscott(f) \leq \dcscottspace(L).$
\end{proof}

\begin{proposition}\label[proposition]{prop: dcscottmap for composition}
    For two simplicial maps $f: K \to L$ and $g: L \to M$, we have $$\dcscott(g \circ f) \leq \mathrm{min}\{\dcscott(f), \dcscott(g)\}.$$
\end{proposition}
\begin{proof}
   Assume that $\dcscott(f)=n$ and the $f$-Farber subcomplexes $\Omega_0,\dots,\Omega_n$ form a cover of $K \Pi K$. Our goal is to show $\Omega=\Omega_j$ is a $g\circ f$-Farber subcomplex for $j=0, \dots, n$, which would imply $\dcscott(g \circ f) \leq n.$
     \[\begin{tikzcd}
   & & L  \arrow[d, "\Delta_L" ]\arrow[r, "g"] & M\arrow[d, "\Delta_M"] \\%
 \Omega \arrow[r, hook, "i"] & K \Pi K \arrow[r, "f^2"] \arrow[ru, "\sigma|_{\Omega}"] & L \Pi L \arrow[r, "g^2"]\arrow[ru, "\sigma'|_{\Omega'}"] & M \Pi M.
\end{tikzcd}\]
By assumption, we have $\Delta_L \circ \sigma \sim f^2|_{\Omega}$ Then we have 
\begin{align*}
    & g^2 \circ \Delta_L \circ \sigma \sim (g \circ f)^2|_{\Omega}\\
     \implies &\Delta_M \circ g \circ  \sigma \sim (g \circ f)^2|_{\Omega}.
\end{align*}
Thus, $g \circ \sigma : \Omega \to M$ and the above contiguity relation imply that $\Omega$ is a $(g \circ f)$-Farber subcomplex. Now suppose that $\dcscott(g)=m$ and the $g$-Farber subcomplexes $\Omega'_1, \dots , \Omega'_m$ cover $L \Pi L$. Our goal is to show that $\Omega= (f^2)^{-1}(\Omega_j)$ is a $(g \circ f)$-Farber subcomplex for $j=0, \dots, m$, which implies that $\dcscott ( g \circ f) \leq m.$

Note that $\Delta_M \circ \sigma' \sim g^2|_{\Omega'}$. Now $\Delta_M \circ (\sigma' \circ f^2|_{\Omega}) \sim g^2|_{\Omega'} \circ f^2|_{\Omega} = (g\circ f)^2|_{\Omega}$. Thus, $\Omega$ is a $(g \circ f)$-Farber subcomplex and we are done.
\end{proof}
\begin{proposition}
    Suppose $h: K \to L$ is any simplicial map. Then
    \begin{enumerate}
        \item if $h$ has a right contiguity inverse, then $\dcscott(h)=\dcscottspace(L);$
         \item if $h$ has a left contiguity inverse, then $\dcscott(h)=\dcscottspace(K);$
          \item if $h$ is a contiguity equivalence, then $\dcscott(h)=\dcscottspace(K)=\dcscottspace(L).$
    \end{enumerate}
\end{proposition}
\begin{proof}
    (1) Suppose that $g$ is the right contiguity inverse of $h$. In other words $h \circ g \sim id_L$. Then by \Cref{prop: contiguity invariance of dcscottmap}, we have $\dcscott(h \circ g)=\dcscott(id_L)$. But $\dcscott(id_L)=\dcscottspace(L)$, and by \Cref{prop: dcscottmap for composition} $\dcscott(h \circ g) \leq \dcscott(h)$. Thus, $\dcscott(h) \geq \dcscottspace(L)$. Again by \Cref{prop: dcscottmap bounded by dcspaces} we have $\dcscott(h) \leq \dcscottspace(L)$ and hence the equality follows. 

    (2) Suppose that $f$ is the left contiguity inverse of $h$, so that $f \circ h \sim id_K$. Then by exactly similar analysis to above, we get $$\dcscottspace(K)=\dcscott(id_K)=\dcscott(f \circ h) \leq \dcscott(h) \leq \dcscottspace(K).$$
    Hence $\dcscott(h) =\dcscottspace(K).$

    (3) This follows as a corollary to the first two parts.
\end{proof}
We now establish an analog of \cite[Theorem 3.4]{fernandez2018discrete} in the context of $f$-Farber subcomplexs, where $f$ is a simplicial map.
\begin{theorem} \label{Farber subcomplex and projection maps}
    Let $\Omega \subset K^2$ be a subcomplex of the categorical product. The following conditions are equivalent:
    \begin{enumerate}
      \item $\Omega$ is an $f$-Farber subcomplex,
        \item the restrictions to $\Omega$ of the compositions of the projections with $f^2$ are in the same contiguity class, i.e., $(\pi_1 \circ f^2)|_{\Omega} \sim(\pi_2 \circ f^2)|_{\Omega}$. 
        
    \end{enumerate}
        \end{theorem}
        \begin{proof}
            $(1) \implies (2)$ If $\Omega \subset K^2$ is an $f$-Farber subcomplex, then there exists a simplicial map $\sigma: \Omega \rightarrow L$ such that $\Delta \circ \sigma \sim f^2|_{\Omega}$. But $\Delta \circ \sigma$ is the map $(\sigma, \sigma)$ defined by $\omega \mapsto (\sigma(\omega), \sigma(\omega))$ for $\omega \in \Omega$. On the other hand $f^2|_{\Omega}=(\pi_1 \circ f^2|_{\Omega}, \pi_2 \circ f^2|_{\Omega})$. Then
            $$(\sigma, \sigma) \sim (\pi_1 \circ f^2|_{\Omega}, \pi_2 \circ f^2|_{\Omega}),$$
            which implies that
            $$(\pi_1 \circ f^2)|_{\Omega} \sim \sigma \sim (\pi_2 \circ f^2)|_{\Omega}.$$

           $(2) \implies (1)$
           If $(\pi_1 \circ f^2)|_{\Omega} \sim (\pi_2 \circ f^2)|_{\Omega}$, define $\sigma: \Omega \rightarrow L$ by $\sigma=(\pi_1 \circ f^2)|_{\Omega}$. 
           Then \begin{align*}
               \Delta_L \circ \sigma =& (\pi_1 \circ (f^2)|_{\Omega}, \pi_1 \circ (f^2)|_{\Omega})\\
               \sim& (\pi_1 \circ (f^2)|_{\Omega}, \pi_2 \circ (f^2)|_{\Omega}) (\text{~by hypothesis~})\\
               =&(f^2)|_{\Omega}.
           \end{align*}
        \end{proof}
\subsection{Relationship with Simplicial LS-category}
    In this subsection, we generalize the inequalities $\scat(K) \leq \dcscottspace(K)\leq \scat(K^2)$ proved in \cite{fernandez2018discrete} for a simplicial complex $K$ to simplicial maps. 
\begin{theorem}
    For any map $f:K \rightarrow L$ we have 
    $$\scat(f) \leq \dcscott(f).$$
\end{theorem}
\begin{proof}
    If $\dcscott(f) \leq n$, let $K^2= \Omega_0 \cup \cdots \cup \Omega_n$ be a covering by $f$-Farber subcomplexes. Fix a base point $v_0 \in K$ and let $\iota_0:K \rightarrow K^2$ be the simplicial map $\iota_0(w)=(v_0, w)$. Then, let us take the inverse images 
    $$J_l=(\iota_0)^{-1}(\Omega_l), ~l=0, \dots,n.$$
    Since $K= J_0 \cup \cdots \cup J_n$, if we prove that each $J_l$ is an $f$-categorical subcomplex, then we can conclude that $\scat(f) \leq n$, and the result follows.

    Let $\Omega \subset K^2$ be an $f$-Farber subcomplex, with a local section $\sigma: \Omega \rightarrow L$, and let
    $$J=(\iota_0)^{-1}(\Omega) \subset K.$$
    We will show that the restriction map $f|_J: J \rightarrow L$ belongs to the same contiguity class of the constant map $f(v_0): J \rightarrow L$, proving that $J$ is an $f$-categorical subcomplex of $K$.

    Since, $\Delta_L \circ \sigma \sim f^2|_{\Omega}$, there is a sequence of simplicial maps $\psi_i: \Omega \rightarrow L^2, ~i=1, \dots, m$, such that
    $$\Delta_L \circ \sigma= \psi_1 \sim_c \dots \sim_c \psi_m=f^2|_{\Omega}.$$
    Then, by composition,
    \begin{align} \label{pr_1}
       \pi_1 \circ \psi_1 \circ \iota_0|_J \sim_c \dots \sim_c \pi_1 \circ \psi_m \circ \iota_0|_{J}, 
    \end{align}
    where, for every $w \in J$, we have
    $$\pi_1 \circ \psi_1 \circ \iota_0 (w)= \pi_1 \circ \Delta_L \circ \sigma (v_0, w)=\sigma(v_0, w),$$
    and $$\pi_1 \circ \psi_m \circ \iota_0 (w)= \pi_1 \circ f^2|_{\Omega}(v_0, w)=f(v_0).$$
    On the other hand 
    \begin{align} \label{pr_2}
        \pi_2 \circ \psi_1 \circ \iota_0|_J \sim_c \dots \sim_c \pi_2 \circ \psi_m \circ \iota_0|_J,
    \end{align}
    where, for every $w \in J$, we have
    $$\pi_2 \circ \psi_1 \circ \iota_0 (w)= \pi_2 \circ \Delta_L \circ \sigma (v_0, w)=\sigma(v_0, w),$$
    and $$\pi_2 \circ \psi_m \circ \iota_0 (w)= \pi_2 \circ f^2|_{\Omega}(v_0, w)=f(w).$$
    Thus, using Theorem \ref{Farber subcomplex and projection maps}, from (\ref{pr_1}) and (\ref{pr_2}) it follows
    $$f(v_0) \sim \sigma(v_0, w) \sim f(w), ~~\mbox{ for every}~ w \in J,$$
    or equivalently, $f(v_0) \sim f|_J$. Hence, $J$ is an $f$-categorical subcomplex.
\end{proof}
We recall \cite[Lemma 4.4]{fernandez2018discrete}, which will be used to prove the next inequality.
\begin{lemma} \label{edge-path}
    An abstract simplicial complex $L$ is edge-path connected if and only if two arbitrary constant maps $K \to L$ are in the same contiguity class.
\end{lemma}
\begin{theorem} \label{DS(f)-scat(f^2) Scott}
    Let $f: K \to L$ be a simplicial map such that $L$ is edge-path connected. Then  $$\dcscott(f) \leq \scat(f^2).$$
\end{theorem}
\begin{proof}
    Assume that $\scat(f^2)=n$ and let $K^2= \Omega_0 \cup \cdots \cup \Omega_n$ be an $f^2$-categorical covering of $K^2$. Our goal is to show that $\Omega_j$ is an $f$-Farber subcomplex for $j=0,\dots,n$, which would imply $\dcscott(f)\leq n$, proving the theorem. 
    
    Set $\Omega:=\Omega_j$ for some $j$.
    By our assumption, $f^2|_{\Omega} \sim *,$ where $*: \Omega \to L^2$ is some constant map whose image is $\{(v_0,w_0)\}.$ Note that the edge-path connectedness of $L$ allows us to choose $*$ such that $v_0=w_0.$ Now by the contiguity relation, there is a sequence of simplicial maps, each one contiguous to the next one (using Lemma \ref{edge-path}). 
    $$f^2|_{\Omega}= \psi_1 \sim_c \cdots \sim_c \psi_m=(v_0,v_0),$$
    where $\psi_j: \Omega \to L^2$. If $\pi_j: L^2 \to L$ are the projection maps onto the $j$-th factor for $j=1,2$, then $\pi_j \circ \psi_k: \Omega \to L$ is contiguous to $\pi_j \circ \psi_{k+1}$. 
    Thus we have $$\pi_1 \circ f^2|_{\Omega} \sim \pi_1 \circ \psi_m = v_0.$$
    Similarly for the second projection map we get $$\pi_2 \circ f^2|_{\Omega} \sim \pi_2 \circ \psi_m = v_0.$$
   This gives us  $\pi_1 \circ f^2|_{\Omega} \sim \pi_2 \circ f^2|_{\Omega}$. From Theorem \ref{Farber subcomplex and projection maps}, we conclude that $\Omega$ is an $f$-Farber subcomplex. 
\end{proof}

We now present an immediate consequence of Theorem \ref{DS(f)-scat(f^2) Scott}.
\begin{corollary}
    A simplicial map $f: K \to L$ is nullcontiguous if and only if $\dcscott(f)=0.$
\end{corollary}

\subsection{Comparison with the TC of the geometric realization}
Given a simplicial map $f: K \to L$, we have a corresponding induced map on the geometric realization, denoted by $|f |: |K| \to |L|.$ In \cite[Theorem 5.2]{fernandez2018discrete}, the authors have shown that the discrete TC of abstract simplicial complexes bounds the TC of its geometric realization. In this section, we generalize their results to simplicial maps. Before we present this generalization, we recall the following homotopy equivalence from \cite[Lemma 5.1]{fernandez2018discrete}.
\begin{lemma}\label{lemma: |K|*|K| htopy eq |K^2|}
    There exists a homotopy equivalence $u : |K| \times |K| \to |K^2|$ satisfying the property that (upto homotopy) $$|\pi_i |\circ u = p_i ~\text{for}~ i=1,2,$$ where $p_i: |K| \times |K| \to |K|$ and $\pi_i: K^2 \to K$ are projection maps.
\end{lemma}
\begin{theorem}
    For a simplicial map $f :  K \to L$, we have $\Tcscott( |f|) \leq \dcscott(f).$ 
\end{theorem} 
\begin{proof}
    Let $\dcscott(f) \leq n$ and let the $f$-Farber subcomplexes $\Omega_1, \dots,\Omega_n$ form a cover for $K \Pi K.$ Then for each of the $f$-Farber subcomplexes $\Omega = \Omega_j$ we have $\Delta_L \circ \sigma \sim f^2|_{\Omega}.$ We have a map $|f^2 |:  |K^2 |\to |L^2|$ and by the construction of geometric realization, $f^2|_{\Omega}$ is the map $|f^2||_{|\Omega|}.$  Since contiguous maps induce continuous homotopic maps, applying the geometric realization functor yields homotopic maps $|\pi_1 |\circ |f^2||_{|\Omega|} \simeq |\pi_2 |\circ |f^2||_{|\Omega|}$. Consider the closed subspaces $F:=F_j:=u^{-1}(\Omega) \subset |K| \times |K|$, which form a closed covering of $|K| \times |K|.$ Then by \Cref{lemma: |K|*|K| htopy eq |K^2|} we have the following equality, upto homotopy $$p_1 \circ (|f^2|_F) = |\pi_1 |\circ u \circ (|f^2|_F)$$ but $u \circ (|f^2|_F) = |f^2||_{u(F)} = |f^2|_{\Omega}$. Hence,  $p_1 \circ (|f^2|_F) = |\pi_1 |\circ |f^2|_{\Omega}$. Similarly, one obtains $p_2 \circ (|f^2|_F) = |\pi_2 |\circ |f^2|_{\Omega}$. We have $\pi_1 \circ f^2|_{\Omega} \sim \pi_2 \circ f^2|_{\Omega}$ by \Cref{Farber subcomplex and projection maps}. Therefore, $p_1 \circ (|f^2|_F)$ is homotopic to $p_2 \circ (|f^2|_F)$. We conclude that $\Tcscott( |f|) \leq n.$ 
\end{proof}

\subsection{Examples}
\begin{example}
    Let $f : K \to L$ be a simplicial map. If either of $K$ or $L$ is a tree, then we have $\dcscott(f)=0.$ More generally, this is true if either of $K$ or $L$ is strongly collapsible.
\end{example}

\begin{example}
\normalfont{
    Consider the two-sheeted covering map $p: S^1 \to \mathbb{R}P^1$ given by identifying the antipodal points. We compute $\dcscott(p)$. First, we obtain an upper bound using \Cref{prop: dcscottmap bounded by dcspaces}. We consider the following simplicial models for $S^1$ and $\mathbb{R}P^1.$
    
    Note that one can choose a $2n$-gon and $n$-gon model for $S^1$ and $\mathbb{R}P^1$, respectively. For the latter case, it gives rise to an $(n+1) \times (n+1)$ grid to represent $\mathbb{R}P^1 \Pi \mathbb{R}P^1$, whose geometric realization is homotopy equivalent to a torus. In the figure below it is a $5 \times 5$ grid and we describe the three $p$-Farber subcomplexes $\Omega_1, \Omega_2$ and $\Omega_3$, implying $\dcscott(p)\leq 2.$ Clearly, $\dcscott(p)$ has to be either $1$ or $2$, as $p$ is not nullhomotopic.

\providecommand{\xs}{\tiny}

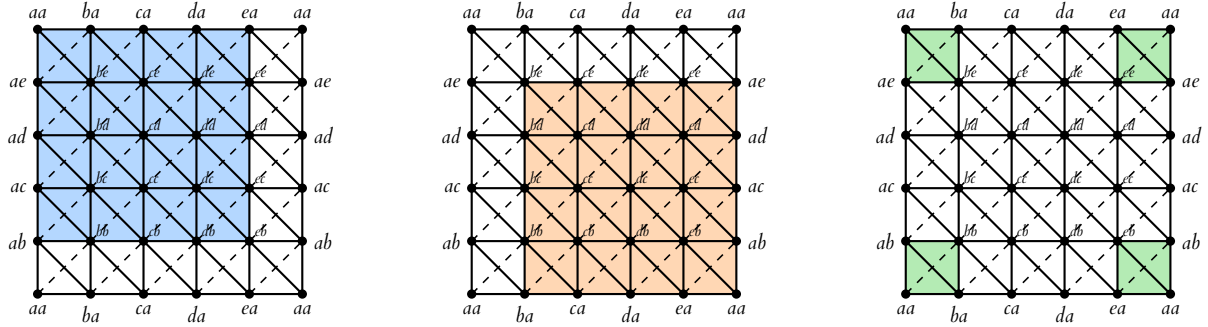
\begin{figure}[htbp]
\centering
\begin{tikzpicture}[scale=0.7, transform shape]

  \begin{scope}[shift={(0,0)}]
    \node[above right, font=\Large\bfseries] at (-0.3, 5.2)  {};
    
    \foreach \x in {0,1,2,3} {
      \foreach \y in {1,2,3,4} {
        \fill[omegaOne] (\x,\y) rectangle (\x+1,\y+1);
      }
    }
  \end{scope}

  \begin{scope}[shift={(8.2,0)}]
    \node[above right, font=\Large\bfseries] at (-0.3, 5.2)  {};
    
    \foreach \x in {1,2,3,4} {
      \foreach \y in {0,1,2,3} {
        \fill[omegaTwo] (\x,\y) rectangle (\x+1,\y+1);
      }
    }
  \end{scope}

  \begin{scope}[shift={(16.4,0)}]
    \node[above right, font=\Large\bfseries] at (-0.3, 5.2)  {};
    
    \foreach \x/\y in {0/0, 4/0, 0/4, 4/4} {
      \fill[omegaThree] (\x,\y) rectangle (\x+1,\y+1);
    }
  \end{scope}

  \foreach \xShift in {0, 8.2, 16.4} {
    \begin{scope}[shift={(\xShift,0)}]
      
      \foreach \x in {0,...,4} {
        \foreach \y in {0,...,4} {
          \draw[thick] (\x, \y+1) -- (\x+1, \y);
          \draw[dashed, semithick] (\x, \y) -- (\x+1, \y+1);
        }
      }

      \draw[thick] (0,0) grid (5,5);

      \foreach \x in {0,...,5} {
        \foreach \y in {0,...,5} {
          \fill[black] (\x,\y) circle (2.5pt);
        }
      }

      \node[above=2pt, font=\small\itshape] at (0,5) {aa};
      \node[above=2pt, font=\small\itshape] at (1,5) {ba};
      \node[above=2pt, font=\small\itshape] at (2,5) {ca};
      \node[above=2pt, font=\small\itshape] at (3,5) {da};
      \node[above=2pt, font=\small\itshape] at (4,5) {ea};
      \node[above=2pt, font=\small\itshape] at (5,5) {aa};

      \node[below=2pt, font=\small\itshape] at (0,0) {aa};
      \node[below=2pt, font=\small\itshape] at (1,0) {ba};
      \node[below=2pt, font=\small\itshape] at (2,0) {ca};
      \node[below=2pt, font=\small\itshape] at (3,0) {da};
      \node[below=2pt, font=\small\itshape] at (4,0) {ea};
      \node[below=2pt, font=\small\itshape] at (5,0) {aa};

      \node[left=2pt, font=\small\itshape] at (0,1) {ab};
      \node[left=2pt, font=\small\itshape] at (0,2) {ac};
      \node[left=2pt, font=\small\itshape] at (0,3) {ad};
      \node[left=2pt, font=\small\itshape] at (0,4) {ae};

      \node[right=2pt, font=\small\itshape] at (5,1) {ab};
      \node[right=2pt, font=\small\itshape] at (5,2) {ac};
      \node[right=2pt, font=\small\itshape] at (5,3) {ad};
      \node[right=2pt, font=\small\itshape] at (5,4) {ae};

      \node[above right=-2pt, font=\xs\itshape] at (1,1) {bb};
      \node[above right=-2pt, font=\xs\itshape] at (2,2) {cc};
      \node[above right=-2pt, font=\xs\itshape] at (3,3) {dd};
      \node[above right=-2pt, font=\xs\itshape] at (4,4) {ee};
      \node[above right=-2pt, font=\xs\itshape] at (2,1) {cb};
      \node[above right=-2pt, font=\xs\itshape] at (1,2) {bc};
      \node[above right=-2pt, font=\xs\itshape] at (3,2) {dc};
      \node[above right=-2pt, font=\xs\itshape] at (3,1) {db};
      \node[above right=-2pt, font=\xs\itshape] at (3,4) {de};
      \node[above right=-2pt, font=\xs\itshape] at (2,4) {ce};
      \node[above right=-2pt, font=\xs\itshape] at (1,4) {be};
      \node[above right=-2pt, font=\xs\itshape] at (1,3) {bd};
      \node[above right=-2pt, font=\xs\itshape] at (2,3) {cd};
      \node[above right=-2pt, font=\xs\itshape] at (4,1) {eb};
      \node[above right=-2pt, font=\xs\itshape] at (4,2) {ec};
      \node[above right=-2pt, font=\xs\itshape] at (4,3) {ed};

    \end{scope}
  }

\end{tikzpicture}
\caption{Covering  by three p-Farber subcomplexes $\Omega_1, \Omega_2,$ and $\Omega_3$.}
\label{fig:three_omegas_5x5}
\end{figure}

    It remains to show that $\mathbb{R}P^1 \Pi \mathbb{R}P^1$ cannot be covered by two $p$-Farber subcomplexes, say $\Omega$ and $\Omega'$. Consider the map $i_0: \mathbb{R}P^1 \to \mathbb{R}P^1 \Pi \mathbb{R}P^1$ given by $i_0(v)=(v_0,v)$. Since $\Omega$ is a $p$-Farber subcomplex, \Cref{Farber subcomplex and projection maps} implies that $i_0^{-1}(\Omega)$ is a $p$-categorical subcomplex of $\mathbb{R}P^1$. As $\mathbb{R}P^1$ is not strongly collapsible, $i_0^{-1}(\Omega)$ is not the whole of $\mathbb{R}P^1$ and hence $\Omega$ cannot contain three consecutive vertical edges. Similarly, $\Omega'$ cannot contain three horizontal edges. As a consequence, $\Omega$ and $\Omega'$ do not cover $\mathbb{R}P^1 \Pi \mathbb{R}P^1$, concluding $\dcscott(p) \geq 2$.
    }
\end{example}
\section{Murillo-Wu's discrete topological complexity of simplicial maps}\label{section 4 M-W}
In \cite{murillo2021topological}, Murillo--Wu introduced a notion of topological complexity for a continuous map. In this section, we consider its discrete analog in the context of simplicial maps. Analogous to the continuous case, as shown by Scott \cite[Theorem 3.4]{scott2022topological}, we prove that the newly defined discrete topological complexity of a simplicial map coincides with the previously defined notion of discrete topological complexity.

\begin{definition}\label[definition]{defn: Murillo-Wu}
   Let $f \colon K \to L$ be a simplicial map between two simplicial complexes. Then the discrete topological complexity in the sense of Murillo-Wu, denoted by $\dcMoriloWu(f)$,  is the smallest integer $n$ such that there exists a cover  $\{\Omega_0,\dots, \Omega_n\}$ of $K\Pi K$ consisting of subcomplexes with simplicial maps $s_i \colon \Omega_i\to K$ satisfying $f^2 \circ \Delta_K \circ s_i\sim f^2|_{\Omega_i}$ for each $0\leq i\leq n$ (see the following commutative diagram),
 \[\begin{tikzcd}
   & K  \arrow[d, "\Delta_K" ] & \\%
 \Omega_i \arrow[r, hook, "i"] \arrow[ru, bend left, "\exists ~ s_i"] & K \Pi K \arrow[r, "f^2"] & L \Pi L.
\end{tikzcd}\]
\end{definition}

In the following proposition, we relate the contiguity invariants $\dcscott(-)$ with $\dcMoriloWu(-)$ and show that they are equal.
\begin{proposition}\label[proposition]{prop: Sc=MW}
    Let $f: K \to L$ be a simplicial map between two simplicial complexes $K$ and $L.$ Then $$\dcscott(f) =\dcMoriloWu(f).$$
\end{proposition}
\begin{proof}
    We first prove that $\dcscott(f) \leq \dcMoriloWu(f).$ Assume that $\dcMoriloWu(f)=k$ and take $\Omega:=\Omega_i$, where $\Omega_i$ are as in \Cref{defn: Murillo-Wu}. Then we have $f^2 \circ \Delta_K \circ s_i\sim f^2|_{\Omega}.$ Let $s_{\Omega}=f \circ s_i,$ where $s_i$ is as in the diagram below. But then $\Delta_L \circ s_{\Omega}= \Delta_L \circ f \circ s_i = f^2 \circ \Delta_k \circ s_i \sim f^2|_{\Omega}.$ Hence $\dcscott(f) \leq \dcMoriloWu(f).$

    Consider the following diagram:
    \[\begin{tikzcd}
   & K  \arrow[d, "\Delta_K" ]\arrow[r,"f"] & L \arrow[d, "\Delta_L"] \\%
 \Omega_i \arrow[r, hook, "i"] \arrow[ru, bend left, "\exists ~ s_i"] & K \Pi K \arrow[r, "f^2"] & L \Pi L.
\end{tikzcd}\]
For the other direction, suppose there exists an $f$-Farber subcomplex $\Omega$ as in the definition of $\dcscott.$ \Cref{Farber subcomplex and projection maps} then implies that $\pi^L_1 \circ f^2|_{\Omega} \sim \pi^L_2 \circ f^2|_{\Omega}.$ Thus 
    \begin{align*}
        f^2|_{\Omega}  =& (\pi^L_1 \circ f^2|_{\Omega} ,\pi^L_2 \circ f^2|_{\Omega})\\
        \sim & (\pi^L_1 \circ f^2|_{\Omega} ,\pi^L_1 \circ f^2|_{\Omega}) \\
        =&(f \circ \pi^K_1|_{\Omega},f\circ \pi^K_1|_{\Omega})\\
        =& (f,f) \circ \pi^K_1|_{\Omega}\\
        =& \Delta_L \circ f \circ \pi^K_1|_{\Omega}\\
        =& f^2\circ \Delta_K \circ \pi^K_1|_{\Omega}. ~~(\text{Using commutativity of the diagram.})
         \end{align*}
    Hence if we take $s_i:= \pi^K_1|_{\Omega}$, then it satisfies \Cref{defn: Murillo-Wu} and consequently we obtain $\dcscott(f) \geq \dcMoriloWu(f).$ Combining both, we have $\dcscott(f) = \dcMoriloWu(f).$
\end{proof}

\begin{remark}
In view of the Proposition \ref{prop: Sc=MW}, which identifies $\dcscott(f) $ and $\dcMoriloWu(f)$, any discrete analog of the results of Murillo and Wu would yield no additional information in this context. Consequently, we do not pursue a separate discrete development of their theory here, as it is already subsumed by the established equivalence.    
\end{remark}

\section{Acknowledgments}
Sutirtha Datta acknowledges UGC Junior Research Fellowship. Navnath Daundkar gratefully acknowledges the support of the DST–INSPIRE Faculty Fellowship (Faculty Registration No.~IFA24-MA218), Department of Science and Technology, Government of India; the Industrial Consultancy and Sponsored Research (IC\&SR), Indian Institute of Technology Madras, through the New Faculty Initiation Grant (RF25261395MANFIG009294).

\bibliographystyle{plain} 
\bibliography{references}

\end{document}